\documentclass[11pt,a4paper]{amsart}
\usepackage[margin=25mm]{geometry}
\usepackage[numbers]{natbib}
\usepackage{hyperref}

\usepackage{amssymb,amsmath}
\usepackage{amsthm}
\usepackage{bbm, bm}
\usepackage{stmaryrd}
\usepackage[usenames,dvipsnames]{xcolor}
\usepackage{color}
\usepackage{graphicx}
\usepackage{caption}
\usepackage{float}
\usepackage{wrapfig}
\usepackage{enumerate,enumitem}  

\usepackage{lineno}

\usepackage{abstract}

\input xy
\xyoption{all} 

\numberwithin{equation}{section}

\newtheorem{theorem}{Theorem}[section]

\newtheorem{proposition}[theorem]{Proposition}

\theoremstyle{definition}
\newtheorem{definition}[theorem]{Definition}

\newenvironment{remark}[1][Remark.]{\begin{trivlist}
\item[\hskip \labelsep {\bfseries #1}]  }{ \end{trivlist}}

\newcommand{\Id}{\mathbbmss{1}}

\newcommand{\rmd}{\textnormal{d}}

\DeclareMathOperator{\Vect}{Vect}

\DeclareMathOperator{\Span}{Span}

\font\black=cmbx10 \font\sblack=cmbx7 \font\ssblack=cmbx5 \font\blackital=cmmib10  \skewchar\blackital='177
\font\sblackital=cmmib7 \skewchar\sblackital='177 \font\ssblackital=cmmib5 \skewchar\ssblackital='177
\font\sanss=cmss10 \font\ssanss=cmss8 
\font\sssanss=cmss8 scaled 600 \font\blackboard=msbm10 \font\sblackboard=msbm7 \font\ssblackboard=msbm5
\font\caligr=eusm10 \font\scaligr=eusm7 \font\sscaligr=eusm5  \font\fraktur=eufm10
\font\sfraktur=eufm7 \font\ssfraktur=eufm5 
\font\bsymb=cmsy10 scaled\magstep2
\def\all#1{\setbox0=\hbox{\lower1.5pt\hbox{\bsymb
       \char"38}}\setbox1=\hbox{$_{#1}$} \box0\lower2pt\box1\;}
\def\exi#1{\setbox0=\hbox{\lower1.5pt\hbox{\bsymb \char"39}}
       \setbox1=\hbox{$_{#1}$} \box0\lower2pt\box1\;}

\def\tx#1{{\fam0\relax#1}}

\newfam\bifam
\textfont\bifam=\blackital \scriptfont\bifam=\sblackital \scriptscriptfont\bifam=\ssblackital

\newfam\blfam
\textfont\blfam=\black \scriptfont\blfam=\sblack \scriptscriptfont\blfam=\ssblack

\newfam\bbfam
\textfont\bbfam=\blackboard \scriptfont\bbfam=\sblackboard \scriptscriptfont\bbfam=\ssblackboard

\newfam\ssfam
\textfont\ssfam=\sanss \scriptfont\ssfam=\ssanss \scriptscriptfont\ssfam=\sssanss
\def\sss#1{{\fam\ssfam\relax#1}}

\newfam\clfam
\textfont\clfam=\caligr \scriptfont\clfam=\scaligr \scriptscriptfont\clfam=\sscaligr

\newfam\frfam
\textfont\frfam=\fraktur \scriptfont\frfam=\sfraktur \scriptscriptfont\frfam=\ssfraktur

\def\hpb#1{\setbox0=\hbox{${#1}$}
    \copy0 \kern-\wd0 \kern.2pt \box0}
\def\vpb#1{\setbox0=\hbox{${#1}$}
    \copy0 \kern-\wd0 \raise.08pt \box0}

\def\pmb#1{\setbox0\hbox{${#1}$} \copy0 \kern-\wd0 \kern.2pt \box0}
\def\pmbb#1{\setbox0\hbox{${#1}$} \copy0 \kern-\wd0
      \kern.2pt \copy0 \kern-\wd0 \kern.2pt \box0}
\def\pmbbb#1{\setbox0\hbox{${#1}$} \copy0 \kern-\wd0
      \kern.2pt \copy0 \kern-\wd0 \kern.2pt
    \copy0 \kern-\wd0 \kern.2pt \box0}
\def\pmxb#1{\setbox0\hbox{${#1}$} \copy0 \kern-\wd0
      \kern.2pt \copy0 \kern-\wd0 \kern.2pt
      \copy0 \kern-\wd0 \kern.2pt \copy0 \kern-\wd0 \kern.2pt \box0}
\def\pmxbb#1{\setbox0\hbox{${#1}$} \copy0 \kern-\wd0 \kern.2pt
      \copy0 \kern-\wd0 \kern.2pt
      \copy0 \kern-\wd0 \kern.2pt \copy0 \kern-\wd0 \kern.2pt
      \copy0 \kern-\wd0 \kern.2pt \box0}

\mathchardef\za="710B  
\mathchardef\zb="710C  
\mathchardef\zg="710D  
\mathchardef\zd="710E  
\mathchardef\zve="710F 
\mathchardef\zz="7110  
\mathchardef\zh="7111  
\mathchardef\zvy="7112 
\mathchardef\zi="7113  
\mathchardef\zk="7114  
\mathchardef\zl="7115  
\mathchardef\zm="7116  
\mathchardef\zn="7117  
\mathchardef\zx="7118  
\mathchardef\zp="7119  
\mathchardef\zr="711A  
\mathchardef\zs="711B  
\mathchardef\zt="711C  
\mathchardef\zu="711D  
\mathchardef\zvf="711E 
\mathchardef\zq="711F  
\mathchardef\zc="7120  
\mathchardef\zw="7121  
\mathchardef\ze="7122  
\mathchardef\zy="7123  
\mathchardef\zf="7124  
\mathchardef\zvr="7125 
\mathchardef\zvs="7126 
\mathchardef\zf="7127  
\mathchardef\zG="7000  
\mathchardef\zD="7001  
\mathchardef\zY="7002  
\mathchardef\zL="7003  
\mathchardef\zX="7004  
\mathchardef\zP="7005  
\mathchardef\zS="7006  
\mathchardef\zU="7007  
\mathchardef\zF="7008  
\mathchardef\zW="700A  
\mathchardef\zC="7009  

\newcommand{\be}{\begin{equation}}
\newcommand{\ee}{\end{equation}}

\newcommand{\bea}{\begin{eqnarray}}
\newcommand{\eea}{\end{eqnarray}}
\def\*{{\textstyle *}}
\newcommand{\R}{{\mathbb R}}

\newcommand{\s}{{\textstyle *}}

\def\Sec{\sss{Sec}}
\def\Vect{\sss{Vect}}

\def\sH{{\sss H}}

\def\sV{{\sss V}}

\def\xi{\tx{i}}

\def\s*{{\scriptstyle *}}

\newcommand{\beas}{\begin{eqnarray*}}
\newcommand{\eeas}{\end{eqnarray*}}

\title{Carrollian $\R^\times$-bundles III: The Hodge Star and Hodge--de Rham Laplacians} 
\author{Andrew James Bruce } 
   \email{andrewjamesbruce@googlemail.com}

   \date{\today}
 
\begin{document}
 \maketitle
\vspace{-20pt}
\begin{abstract}{\noindent Carrollian $\R^\times$-bundles ($\R^\times := \R\setminus \{0\}$) offer a novel perspective on intrinsic Carrollian geometry using the powerful tools of principal bundles. Given a choice of principal connection, a canonical Lorentzian metric exists on the total space. This metric enables the development of Hodge theory on a Carrollian $\R^\times$-bundle; specifically, the Hodge star operator and Hodge--de Rham Laplacian are constructed. These constructions are obstructed on a Carrollian manifold due to the degenerate metric. The framework of Carrollian $\R^\times$-bundles bridges the gap between Carrollian geometry and (pseudo)-Riemannian geometry.  As an example, the question of the Hodge--de Rham Laplacian on the event horizon of a Schwarzschild black hole is addressed. A Carrollian version of electromagnetism is also proposed. }
\\
\noindent {\Small \textbf{Keywords:} Carrollian geometry;~Hodge theory;~Hodge--de Rham Laplacian;~Black hole horizons}

\end{abstract}
\tableofcontents
\section{Introduction}
The intrinsic definition of a \emph{Carrollian manifold} was first given by Duval et al.\ \cite{Duval:2014a,Duval:2014b, Duval:2014} as a smooth manifold equipped with a degenerate metric whose kernel is generated by a complete, nowhere vanishing vector field. Earlier works that inspired the intrinsic definition include Lévy-Leblond \cite{Lévy-Leblond:1965}, Sen Gupta \cite{SenGupta:1966}, and Henneaux \cite{Henneaux:1979}. Natural examples of Carrollian manifolds include null hypersurfaces and the event horizon of a Schwarzschild black hole. For a review of Carrollian physics the reader may consult Bagchi et al.\ \cite{Bagchi:2025}, and references therein. The theory of differential forms, and in particular the  Hodge--de Rham Laplacian, on a Carrollian manifold is largely missing from the current literature. The reason for this lies in the degeneracy of the metric, preventing a direct analogue of the Hodge star operator. Critically, the Hodge star operator requires a (non-degenerate) metric and an orientation to be defined.  Given the importance of the Hodge--de Rham Laplacian in relativistic theories, for example, the study of gravitational waves, analysing solutions of the Einstein field equations, formulating electromagnetism, etc., it is desirable to get a handle on the Carrollian analogue of the Laplacian. \par 
This note continues the author's exploration of \emph{Carrollian $\R^\times$-bundles} ($\R^\times = \R \setminus \{0\}$) as introduced in \cite{Bruce:2025}. A \emph{Carrollian $\R^\times$-bundle} is a principal $\R^\times$-bundle $\pi :P\rightarrow M$ equipped with a degenerate metric $g$ such that $\ker(g) := \left \{X \in \Vect(P)  ~|~  g(X,-)=0 \right \} =  \Sec(\sV P)$, where $\sV P$ is the vertical bundle. Note that, as we have a principal bundle, the vertical bundle is trivial even if the principal bundle itself is non-trivial.  Such bundles offer a novel perspective to describe aspects of intrinsic Carrollian geometry. Any $\R^\times$-bundle can be equipped with a principal connection. For the case of a Carrollian $\R^\times$-bundles, fixing a connection canonically defines a Lorentzian metric on the total space. The theory closely mimics Kaluza--Klein geometry, though the extra dimension is not compact. None-the-less, we can apply the standard tools of (pseudo-)Rienannian geometry to Carrollian geometry, albeit at the cost of introducing a principal connection.\par 
Recall that a \emph{Carrollian bundle} is a triple $(E, g, \kappa)$, where $\pi :E\rightarrow M$ is a fibre bundle with typical fibre $\R$, $g$ is a degenerate metric on $E$ of signature $(1,1, \cdots, 1, 0)$ (the zero is in the fibre position), and $\kappa \in \Vect(E)$ is a complete and non-singular Killing vector field, such that $\ker(g) = \Span \{\kappa\}$. The selection of a section of $E$ induces a fibre-preserving diffeomorphism that can be composed with the smooth inclusion $P \hookrightarrow L$ (see \cite{Bruce:2025} for details)
$$P \hookrightarrow L \stackrel{\sim}{\rightarrow} E\,,$$
where $L$ is a line bundle (over $M$), and $P = L^\times = L\setminus \{0_M\}$ is the associated $\R^\times$-principal bundle. An important result is \cite[Theorem 3.20]{Bruce:2025}, which states that one can build a Carrollian $\R^\times$-bundle structure on $P$ from a Carrollian bundle, albeit non-canonically. Thus, it is generally sufficient to consider $L$ (the linearisation of $E$ about a section) and the principal bundle $P$ while discussing intrinsic Carrollian geometry. However, care is needed when passing back from the principle bundle to the line bundle as complications can arise with the smooth inclusion; locally this means ``reinstating'' $t =0$.\par
In this note, we use the canonical Lorentzian structure to build the Hodge operators, so the Hodge star operator, the de Rham codifferential, and the Hodge--de Rham Laplacian on a Carrollian $\R^\times$-bundle. While the constructions themselves are standard, their application to Carrollian $\R^\times$-bundles yields new novel geometric insight. For on overview of the standard constructions, the reader may  \cite[Chapter 6]{Fecko:2006}, \cite[Chapter 14]{Frankel:2011} or \cite[Section 7.9]{Nakahara:2003}, for example. One of the main observations of this note is that the Hodge operators are equivariant with respect to the $\R^\times$-action; this is due to the precise form of the canonical Lorentzian metric and the assumption that the degenerate metric is constant along the degenerate direction. For the specific example of the event horizon of a Schwarzschild black hole equipped with the trivial connection, we see that the Hodge--de Rham Laplacian can be extended to include $t=0$, and so gives a well-defined Laplacian on differential forms on the event horizon for all $t$, even though the Hodge star operator is ill-defined there, see Subsection \ref{sebsec:BHHorizon} for details including the required regularity condition on differential forms. The event horizon example is generic for trivial Carrollian $\R^\times$-bundles with a two-dimensional base manifold that are equipped with the trivial connection; this is the most physically interesting scenario from the perspective of general relativity, as it covers important cases such as null hypersurfaces, assuming they have a trivial bundle structure. \par 
As another application of the Hodge star operator, we propose a Carrollian version of electromagnetism that mimics the standard geometric construction of Lorentzian electromagnetism. Importantly, the \emph{Carrollian Maxwell equations} \eqref{eqn:CarrollianEM}  extend to cover $t=0$, even though the Hodge star operator is not well-defined on $L$.  That is, the local form of the equations are well-defined for all time, despite there not being a direct geometric formulation of the theory. Moreover, the Carrollian Maxwell equations can be cast into the form of the standard Maxwell equations, but with the temporal parameter being ``logarithmic time'' $u = \ln |t|$.  We stress that the theory obtained is not a ultra-relativistic limit of Lorentzian electromagnetism, and that we have non-trivial dynamics. The physical interpretation of the theory requires further work.  \par
We remark that our approach to the Hodge star operator differs from that of Fecko \cite{Fecko:2023,Fecko:2024}, where a modified Hodge star operator on Carrollian manifolds lacking some standard properties is defined. The Hodge star operator we define exhibits the standard properties, albeit only so away from $t=0$.
\section{Operators on Differential Forms on Carrollian $\R^\times$-bundles}
\subsection{Carrollian $\R^\times$-bundles with a Connection}
The fundamental definition we begin with is the following.
\begin{definition}[\cite{Bruce:2025}]
A principal $\R^\times$-bundle $\pi :P\rightarrow M$ is said to be a \emph{Carrollian $\R^\times$-bundle} if it is equipped with a degenerate metric $g$ such that $\ker(g) := \left \{X \in \Vect(P)  ~|~  g(X,-)=0 \right \} =  \Sec(\sV P)$, where $\sV P$ is the vertical bundle .
\end{definition}
We will take $\dim M = n$, and so $\dim P = n+1$. Note that the fundamental vector field of the principal action $\Delta_P$, which we will refer to as the \emph{Euler vector field}, spans  $\Sec(\sV P)$. We will restrict our attention to Carrollian $\R^\times$-bundles  with a chosen $\R^\times$-connection $(P, g, \Phi)$, such that $M$ is $n$-dimensional, the degenerate metric has signature $(1,\cdots,1,0)$, and the Euler vector field is a Killing vector field. As we will employ local coordinates $(x^a, t)$ on $P$, the signature of the degenerate metric has been chosen to align with the ordering of local coordinates, i.e., the kernel of $g$ lies along the fibre direction. In these adapted coordinates, $\Delta_P = t \partial_t$.  Note that the admissible  fibre preserving coordinate changes on $P$ are ``linearised Carrollian diffeomorphisms'', i.e., $x^{a'} = x^{a'}(x)$ and $t' = \phi(x)t$.  As shown in \cite{Bruce:2025}, any Carrollian bundle, i.e., a Carrollian manifold in which the associated foliation is a fibre bundle with typical fibre $\R$, can be (non-canonically) linearised, and there is a fibre-preserving diffeomorphism from the linearised bundle to the initial Carrollian bundle. Coupled with the one-to-one correspondence between line bundles and principal $\R^\times$-bundle, there is no real loss of generality; though care is needed when extending the constructions to $t=0$. In effect, we are picking adapted coordinates to make the coordinate changes linear in the fibre coordinates. \par 
The connection one-form associated with $\Phi$ is denoted $\theta$, which locally is given by $\theta = t^{-1} \rmd t + \rmd x^a A_a(x)$. Note that $\theta$ is a globally defined, real-valued one-form on $P$, but it is ill-defined on $L$ due to the singular nature at $t =0$. From the definition we immediately have $i_{\Delta_P}\theta = \mathbf{1}$, and $\mathcal{L}_{\Delta_P}\theta  = 0$.  As first presented in \cite{Bruce:2025}, $(P, g, \Phi)$ comes with the non-degenerate Lorentzian metric $G$, defined as $G := g - \theta \otimes \theta$, i.e., the signature of the metric is $(1,\cdots, 1, -1 )$; note the signature is aligned with the ordering of our coordinates on $P$. The connection is seen as part of the underlying geometry and not a dynamical field that is required by local symmetries or sourced by matter - its role is an intrinsic part of the Carrollian geometry.  Written out (locally) in block form, we have
\begin{align*}
G_{ij} = \begin{pmatrix}
(g_M)_{ab}   & 0 \\
0  & - 1
\end{pmatrix}
\,,
& &
G^{ij} = \begin{pmatrix}
(g_M)^{ab}  & 0 \\
0  & - 1
\end{pmatrix}  \, ,
\end{align*}
where $g_M$ is understood as the non-degenerate metric on the base manifold $M$. \par
In the coordinate coframe, locally we have 
\begin{align*}
G = \begin{pmatrix}
g_M - A   A^T & -t^{-1} A \\
-t^{-1}A^T  & - t^{-2}
\end{pmatrix}
\,,
& &
G^{-1} = \begin{pmatrix}
g_M^{-1}  & -t g_M^{-1} A \\
-t A^Tg_M^{-1}  & - t^2+ t^2 A^Tg_M^{-1}A
\end{pmatrix}  \, ,
\end{align*}
where $A$ is the vector of local gauge fields. Note that $\det(G) = - t^{-2}\,\det(g_M) $, and thus $\sqrt{|G|} = \sqrt{|g_M|}\,  t^{-1} $. We then have
\begin{equation}\label{eqn:VolP}
\mathsf{Vol}_P = (-1)^n \theta \wedge \mathsf{Vol}_M \,,
\end{equation}
where $\mathsf{Vol}_M$ is the canonical volume on $M$ built from $g_M$. It is important to note the metric on $M$ and the connection allows for a global invariant measure to be constructed, and this is vital to the construction of the Hodge star operator and the Hodge Laplacian. \par 
A differential $k$-form $\zx \in \Omega^k(P)$ is said to be \emph{homogeneous and of weight} $\lambda$ if $\mathcal{L}_{\Delta_P}\zx = \lambda\, \zx$. A differential $k$-form $\zx \in \Omega^k(P)$ is said to be a \emph{horizontal differential form} if $i_{\Delta_P} \zx = 0$. Locally, a horizontal differential form has no $\rmd t$ component.  We denote the module of horizontal differential $k$-forms as $\Omega_\sH^k(P)$. A differential $k$-form $\beta \in \Omega^k(P)$ is said to be a \emph{vertical differential form} if $i_X \eta = 0$ for all vector fields $X\in \Sec(\sH P)$.  A standard result from the theory of principal bundles is the decomposition of differential $k$-forms
\begin{equation}\label{eqn:DecompForms}
\Omega^k(P) = \Omega_\sH^k(P)\oplus \Omega_\sV^k(P)  = \Omega_\sH^k(P)\oplus \theta \wedge \Omega_\sH^{k-1}(P)\,,
\end{equation}
in particular, the vertical differential $k$-forms are $\Omega_\sV^k(P) \cong \theta \wedge \Omega_\sH^{k-1}(P)$. For example, the \emph{curvature two-form} of the connection $\Phi$ is defined as  $F :=  \rmd \theta \in  \Omega_\sH^2(P)$.  Observe that, the definition of a principal connection, we have  
$$\mathcal{L}_{\Delta_P}(\alpha + \theta \wedge \beta) = \mathcal{L}_{\Delta_P} \alpha + \theta \wedge \mathcal{L}_{\Delta_P} \beta\,,$$
thus, the Lie derivative with respect to the Euler vector field respects the decomposition of differential forms into horizontal and vertical differential forms.
\begin{proposition}
Let $P$ be an $\R^\times$-bundle. Then the de Rham derivative $\rmd : \Omega^k(P) \rightarrow \Omega^{k+1}(P)$ preserves the weight of homogeneous differential forms.
\end{proposition}
\begin{proof}
This is immediate as $\mathcal{L}_{\Delta_P} \circ \rmd =  \rmd \circ \mathcal{L}_{\Delta_P}$.
\end{proof}
\begin{remark}
We will concentrate on the case where the Euler vector field is Killing for the degenerate metric $g$; this is not essential for the following constructions. However, the non-Killing case will lead to the weight of differential forms not being preserved.  
\end{remark}
Note that, as standard, the de Rham differential does not respect the splitting of differential forms into horizontal and vertical differential forms. However, one can construct the \emph{covariant de Rham derivative} as 
\begin{align*}
& D : \Omega^k_\sH(P) \longrightarrow \Omega^{k+1}_\sH(P)\\
& D = \rmd - \theta \wedge \mathcal{L}_{\Delta_P}
\end{align*}
which locally is given by $D = \rmd x^a \left (\frac{\partial}{\partial x^a} - A_a \, t\frac{\partial}{\partial t} \ \right)$.
%
%
\subsection{The Hodge Star operator}
The Hodge star operator on $(P, g, \Phi)$ is defined in the standard way using the Lorentzian metric $G = g - \theta \otimes \theta$. 
\begin{definition}
Let $(P, g, \Phi)$ be a Carrollian $\R^\times$-bundle with a connection. The \emph{Hodge star operator} acting on  a $k$-form $\zx \in \Omega^k(P)$ is defined as the unique $(n+1 -k)$-form $\star \zx$, that satisfies
 $$\eta \wedge \star \zx :=  \textsf{Vol}_P\, \langle \eta, \zx\rangle_G = (-1)^n \,\theta \wedge \textsf{Vol}_ M\, \langle \eta, \zx\rangle_G\,, $$
 for all $k$-forms $\eta$. Here $\langle \cdot, \cdot \rangle_G$ is the pointwise inner product induced by $G$.
\end{definition}
One readily checks that $\star \star \zx = (-1)^{1 +k(n+1 -k)} \, \zx$, taking into account $\dim P = n+1$ and that $G$ has Lorentzian signature. Other standard results are $\star \mathbf{1} = \textsf{Vol}_P$ and $\star \textsf{Vol}_P = - \mathbf{1}$.
\begin{proposition}
Let $(P, g, \Phi)$ be a Carrollian $\R^\times$-bundle with a connection, and $\star$ be the associated Hodge star operator. If $\zx \in \Omega^k(P)$ is homogeneous of weight $\lambda$, then so is $\star \zx$.
\end{proposition}
\begin{proof}
As we have restricted attention to $\Delta_P$ being a Killing vector field for $g$, and $\mathcal{L}_{\Delta_P} \theta =0$, $\Delta_P$ is a Killing vector field for $G = g - \theta \otimes \theta$. We know that the Lie derivative of a Killing vector field commutes with the associated Hodge star operator, i.e.,
$$\mathcal{L}_{\Delta_P}(\star \zx) = \star \mathcal{L}_{\Delta_P} \zx\,.$$
Since $\zx$ is homogeneous, this implies $\mathcal{L}_{\Delta_P}(\star \zx) = \lambda \, \star \zx$.
\end{proof}
Due to the exact form of the induced metric we have the following result which is known from Kaluza--Klein theory.
\begin{proposition}\label{prop:SplitHodgeHV}
Let $(P, g, \Phi)$ be a Carrollian $\R^\times$-bundle with a connection. Then the Hodge star operator maps horizontal differential forms to  vertical differential forms and vice versa, i.e.,
\begin{align*}
& \star : \Omega^k_\sH(P) \longrightarrow  \Omega^{(n+1) -k}_\sV(P)\,, \\
& \star : \Omega^k_\sV(P) \longrightarrow \Omega^{(n+1) -k}_\sH(P)\,,
\end{align*}
and moreover, these maps are isomorphisms.
\end{proposition}
\begin{proof}
Locally we can always find an orthonormal coframe $\{e^a, \theta\}$, such that the the metric takes the form of the Minkowski metric and the volume (locally) simplifies to
$$\mathsf{Vol}_P = e^1 \wedge e^2 \wedge \cdots \wedge e^n \wedge \theta\, .$$  
We set $S_k$ and $S_{k-1}$ to be local differential forms explicitly build using the orthonormal coframe on $M$, i.e., explicitly given in terms of $\{e^a\}$. Then, locally, any differential $k$-forms can be we expressed as $\zx = S_k + \theta \wedge S_{k-1}$ and $\eta = \bar{S}_k + \theta \wedge \bar{S}_{k-1}$.  As the metric is diagonal the calculation of the Hodge star simplifies as compared with using the coordinate basis. Explicitly, 
\begin{align*}
\eta \wedge \star \zx & = (-1)^n \, \theta \wedge \mathsf{Vol}_M \big( \langle \bar{S}_k, S_k \rangle_G + \langle \theta, \theta \rangle_G \langle \bar{S}_{k-1}, S_{k-1} \rangle_G  \big)\\
&= (-1)^n \, \theta \wedge\big(\bar{S}_k \wedge \star_M S_k - \bar{S}_{k-1} \wedge \star_M S_{k-1} \big)\,.
\end{align*}
where $\star_M$ is the Hodge star operator on $M$. We then observe that
$$\star S_k = (-1)^{n+k}\, \theta \wedge \star_M S_k\,, \qquad \star \big(\theta \wedge S_{k-1}\big) = (-1)^{n+1}\, \star_M S_{k-1}\,,$$
In particular, locally and in this chosen cobasis, horizontal differential forms get sent to vertical differential forms, and vice versa. However, this observation does not depend on the chosen local cobasis; though there will not be such a simple relation between $\star$ and $\star_M$. Moreover, as we have a sheaf of differential forms, the global mapping of  horizontal differential forms to vertical differential forms and vertical differential forms to horizontal differential forms is evident. The isomorphism property is evident as $\star \star = \pm \Id$, that is, the Hodge star operator is invertible.
\end{proof}
%
%
\subsection{Towards Carrollian Electromagnetism}
As a potentially physically relevant example, we construct a theory of electromagnetism that closely parallels its standard formulation on curved backgrounds. In particular, we will use the Hodge star operator to formulate a version of electromagnetism on a Carrollian $\R^\times$-bundle and then examine whether the theory can be extended to the associated line bundle $L$.\par 
Consider a Carrollian $\R^\times$-bundle $(P,g, \Phi)$ equipped with a connection, whose connection one-form is $\theta = \rmd t\, t^{-1} + \rmd x^a A_a(x)$, and such that $\Delta_P$ is Killing. We define the Lorentzian metric a $G = g - \theta \otimes \theta$ and consider the associated Hodge star operator. Recall that we have the decomposition of differential $k$-forms into horizontal and vertical parts, i.e., $\Omega^k(P) = \Omega^k_\sH(P) \oplus \theta \wedge \Omega^{k-1}_\sH(P)$.
\begin{definition}
Let $(P,g, \Phi)$ be a Carrollian $\R^\times$-bundle $(P,g, \Phi)$ equipped with a connection and whose Euler vector field is Killing. Then the \emph{Carrollian electromagnetic field strength} on $(P,g, \Phi)$ is a two-form
 $$\mathbb{F} = \mathbb{B} + \theta \wedge \mathbb{E} \in \Omega^2(P)\,, $$ 
where the \emph{magnetic field} is $\mathbb{B} \in \Omega_\sH^2(P)$, and the \emph{electric field} is $\mathbb{E} \in \Omega_\sH^1(P)$. The associated \emph{Carrollian Maxwell equations} (in vacuum) are
\begin{equation}\label{eqn:CarrMax}
\rmd \mathbb{F} =0\,, \qquad \rmd \star \mathbb{F} =0\,.
\end{equation}
\end{definition}
Note that the Carrollian Maxwell equations are fully covariant as they are independent of any choice of coordinates.  Using a local orthonormal cobasis $\{e^a, \theta\}$, we locally write
 $$\mathbb{F}_{loc} = \mathbb{B}_{loc} + \theta \wedge \mathbb{E}_{loc}\,,$$
 where  $\mathbb{B}_{loc}$ and $ \mathbb{E}_{loc}$ are local differential forms explicitly defined in terms of chosen local cobasis.  We then observe that
 $$\star \mathbb{F}_{loc} = (-1)^{n+1}\, \star_M \mathbb{E}_{loc}  +(-1)^n \,\theta \wedge \star_M \mathbb{B}_{loc}\,,$$
 where $\star_M$ is the Hodge star operator on $M$.  Furthermore the local Carrollian Maxwell equations can be written as 
 \begin{align*}
 & \rmd \mathbb{F}_{loc} = \rmd \mathbb{B}_{loc} + \rmd \theta \wedge \mathbb{E}_{loc} - \theta \wedge \rmd \mathbb{E}_{loc}\,, \\
 & \rmd \star \mathbb{F}_{loc} = (-1)^{n+1}\,\rmd (\star_M \mathbb{E}_{loc}) + (-1)^n\, \rmd \theta \wedge \star_M \mathbb{B}_{loc} - (-1)^n \, \theta \wedge \rmd (\star_M \mathbb{B}_{loc})\,.
 \end{align*}
Provided the connection is flat, i.e., $\rmd \theta =0$ (which implies that $P = M \times \R^\times$), the local Carrollian Maxwell equations can be written as
\begin{align}\label{eqn:LocCarrMax}
& \rmd \mathbb{B}_{loc} = 0\,, && \rmd \mathbb{E}_{loc} = 0\,,\\ \nonumber 
& \rmd \star_M \mathbb{B}_{loc} = 0\,, && \rmd \star_M \mathbb{E}_{loc} = 0\,.
\end{align}
Some comments: The local Maxwell equations \eqref{eqn:LocCarrMax} are
\begin{enumerate}
\item  independent of the chosen flat connection;
\item  extend (locally) to include $t =0$, that is they are defined locally on $L$, although the Hodge star operator is degenerate.  There is not a global covariant formulation on $L$.
\end{enumerate}
For simplicity, we consider $P = \R^3 \times \R^\times$, where we use the trivial connection $\theta = \rmd t\, t^{-1}$, and set the metric on $\R^3$ to be the standard Euclidean metric. The natural cobasis to use here is $\{\rmd x^a, \theta\}$. Note that in this example, we have a global orthonormal cobasis. The volume is given by $\mathsf{Vol}_P =  \rmd x \wedge \rmd y \wedge \rmd z\wedge \theta$. The Hodge star operator acting on two-forms  is given by
\begin{align*}
& \star(\rmd x  \wedge \rmd y) = - \theta \wedge \rmd z\,,  && \star(\rmd x  \wedge \rmd z) =  \theta \wedge \rmd y\,, && \star(\rmd y  \wedge \rmd z) = - \theta \wedge \rmd x\,,  \\
& \star(\theta \wedge \rmd x ) = \rmd y \wedge \rmd z\,,  && \star(\theta \wedge \rmd y  ) =  - \rmd x \wedge \rmd z\,, && \star(\theta  \wedge \rmd z) =  \rmd x \wedge \rmd y\,. 
\end{align*}
To make connection with standard electromagnetism, we define the magnetic and electric fields as 
\begin{align*}
&\mathbb{B} := \rmd x \wedge \rmd y \, B_z - \rmd x \wedge \rmd z \, B_y + \rmd y \wedge \rmd z \, B_x\,,\\
& \mathbb{E} := \theta \wedge \big (\rmd x \, E_x + \rmd y \,E_y + \rmd z  \,E_z  \big)\,.
\end{align*}
Carrollian Maxwell equations are (after multiplication by $t$)
\begin{align}\label{eqn:CarrollianEM}
&\nabla \cdot \vec{B} = 0\,, && \nabla \times \vec{E} - L_{\Delta_P} \vec{B}=0\,,\\\nonumber
& \nabla \cdot \vec{E} = 0 \,, && \nabla \times \vec{B} + L_{\Delta_P} \vec{E}=0 \,, 
\end{align}
where $\Delta_P = t \partial_t$.\par
Some further comments are in order: The Carrollian Maxwell equations \eqref{eqn:CarrollianEM}
\begin{enumerate}
\item are \emph{not} the electric nor the magnetic limit of standard Lorentzian electromagnetism (see \cite{Duval:2014}). 
\item exhibit electromagnetic duality, i.e., $\vec{E} \mapsto \vec{B}$ and $\vec{B} \mapsto -\vec{E}$. 
\item extend to $L$, via the smooth inclusion $P \hookrightarrow L$  even though the Hodge star operator is not well-defined on $L$; it is degenerate at $t=0$. 
\end{enumerate}
The local coordinate form of the Carrollian Maxwell equations are \emph{not} invariant under all admissible coordinate transformations of $P = \R^3 \times \R^\times$ (and by extension on $L = \R^3 \times \R$). However, we observe that the equations are invariant under  the subgroup defined by rescaling of time, i.e., $t\mapsto t' = \phi_0 \, t$, where $\phi_0 \in \R^\times$. This observations aligns with the idea that an observer is `frozen' in the spacial directions, and so the only reasonable transformations that can be allowed are rescaling of time (remembering we have linearised the admissible form of Carrollian diffeomorphisms). \par 
It will be illustrative to consider ``logarithmic time'' $u = \ln|t|$, as the Carrollian Maxwell equations can be written as
\begin{align}
&\nabla \cdot \vec{B} = 0\,, && \nabla \times \vec{E} =  \partial_u \vec{B}\,,\\\nonumber
& \nabla \cdot \vec{E} = 0 \,, && \nabla \times \vec{B} =  - \partial_u \vec{E} \,. 
\end{align}
This reformulation casts the theory into a form closer to standard Lorentzian electromagnetism, with the temporal parameter being $u$.  Following the standard calculations, we arrive at the  \emph{Carrollian  wave equations} 
\begin{align}
\partial^2_u \vec{E}  - \nabla^2 \vec{E} =0 \,,\\ \nonumber
\partial^2_u \vec{B} - \nabla^2 \vec{B}= 0 \,.
\end{align}
Thus, Carrollian electromagnetism admits freely propagating wave solutions where the temporal parameter is ``logarithmic time'' $u$. Whether or not $u$ encodes a physical degree of freedom is not immediately clear. None-the-less, the constructions in this subsection suggest that non-trivial dynamics can exits on Carrollian manifolds.
%
%
\subsection{The  de Rham Codifferential}
With the Hodge star operator in place, the de Rham codifferential is defined as standard.
\begin{definition}
Let $(P, g, \Phi)$ be a Carrollian $\R^\times$-bundle with a connection. The \emph{de Rham codifferential} is the $\R$-linear map
$$\delta: \Omega^k(P) \rightarrow \Omega^{k-1}(P)\,,$$
defined as $\delta \zx := (-1)^{1 + k(n+1-k)}\, \star \rmd \star\zx$, where $\rmd$ is the  de Rham differential on $P$, and $\star$ is the Hodge star operator associated with $G = g - \theta \otimes \theta$.
\end{definition}
A standard result is that
 $$\delta^2 \zx = 0\,,$$
for any $\zx \in \Omega^k(P)$.\par 
As de Rham differential and the Hodge star operator respects homogeneity and weight, the de Rham codifferential behaves similarly. 
\begin{proposition}
Let $(P, g, \Phi)$ be a Carrollian $\R^\times$-bundle with a connection, and let $\delta$ be the associated de Rham codifferential. Then, if $\zx \in \Omega^k(P)$ is a homogeneous  and of weight $\lambda$, then so is $\delta\zx \in \Omega^{k-1}(P)$.
\end{proposition}
\begin{proof}
As $\mathcal{L}_{\Delta_P}$ commutes with the Hodge dual and the covariant de Rham derivative, it is clear that $\mathcal{L}_{\Delta_P}\big(\delta \zx \big) = \delta\big(\mathcal{L}_{\Delta_P}\zx\big)$. Thus, assuming $\mathcal{L}_{\Delta_P} \zx = \lambda \, \zx$, we have $\mathcal{L}_{\Delta_P}\big(\delta \zx \big) = \lambda \, \delta \zx$.
\end{proof}
%
%
\subsection{The Hodge--de Rham Laplacian}
With the de Rham codifferential in place, we can define the associated Hodge--de Rham Laplacian. 
\begin{definition}
Let $(P, g, \Phi)$ be a Carrollian $\R^\times$-bundle with a connection. The \emph{Hodge--de Rham Laplacian} is the $\R$-linear map
$$\Delta_{HdR} : \Omega^k(P) \rightarrow \Omega^{k}(P)\,,$$
defined as $\Delta_{HdR}  := \rmd \circ \delta + \delta \circ \rmd$, where $\rmd$ is the de Rham differential, and $\delta$ is the de Rham codifferential on $P$.
\end{definition}
The Hodge de-Rham Laplacian  does not, in general, respect the decomposition of differential forms into horizontal and vertical differential forms.
\begin{proposition}
Let $(P, g, \Phi)$ be a Carrollian $\R^\times$-bundle with a connection, and $\Delta_{HdR}$ be the associated Hodge--de Rham Laplacian. Then, if $\zx \in \Omega^k(P)$ is a homogeneous and of weight $\lambda$, then so is $\Delta_{HdR}\zx \in \Omega^k(P)$.
\end{proposition}
\begin{proof}
This follows directly as $\mathcal{L}_{\Delta_P}$ commutes with both $\rmd$ and $\delta$, which implies that $\mathcal{L}_{\Delta_P}\circ \Delta_{HdR} = \Delta_{HdR} \circ \mathcal{L}_{\Delta_P}$. 
\end{proof}
\begin{definition}
Let $(P, g, \Phi)$ be a Carrollian $\R^\times$-bundle with a connection. A differential $k$-form $\zx \in \Omega^k(P)$ is said to be
\begin{enumerate}
\item \emph{closed} if $\rmd \zx =0$,
\item \emph{coclosed} if $\delta \zx =0$, and
\item \emph{harmonic} if $\Delta_{HdR} \zx =0$.
\end{enumerate}
\end{definition}
Note that if a differential $k$-form is both closed and coclosed, then it is harmonic. However, as we have a Lorentzian metric $G = g - \theta \otimes \theta$, the converse is not necessarily true.
\subsection{Harmonic Forms on Event Horizons}\label{sebsec:BHHorizon}
A potentially physically interesting application of  the notions established in this note is the following.  Consider the Schwarzschild black hole $(\mathcal{M} = \R^2 \times S^2, \rmd s^2)$, where we will employ Eddington--Finkelstein coordinates to write the spacetime interval as
$$\rmd s^2 = - \left(1 - \frac{(2 \kappa)^{-1}}{r} \right)\rmd v^2 + 2\, \rmd v \rmd r + r^2 \rmd \Omega^2\,,$$
where $\rmd \Omega^2$ is the round metric on the sphere $S^2$. Here $\kappa$ is the surface gravity of the black hole.  The \emph{event horizon} is defined as $\mathcal{H} := \{ p\in \mathcal{M} ~~|~~ r(p) = (2 \kappa)^{-1}\} = S^2 \times \R$, which is a trivial line bundle. The induced degenerate metric on $\mathcal{H}$ is $g = (2 \kappa)^{-2} g_{S^2}$, where $g_{S^2}$  is the round metric on $S^2$, note the signature is $(1,1,0)$. The degenerate direction is spanned by the Killing vector field $\partial_v$, and so $\ker(g) = \Span(\partial_v)$. \par 
The principal  $\R^\times$-bundle is given by removing the zero section, and the degenerate metric is $g$ now restricted to $v \neq 0$. To remain consistent with earlier notation, we relabel the fibre coordinate $v$ as $t$.  The Euler vector field $\Delta_P =  t \partial_t$ spans $\ker(g)$; there is no dependence of $t$ in the degenerate metric $g$. Thus, we have a Carrollian $\R^\times$-bundle whose Euler vector field is Killing.  As $P = S^2 \times \R^\times$ is trivial, we can select the trivial connection, so $A =0$ meaning $\theta = \rmd t t^{-1}$.   The metric is then $G =(2 \kappa)^{-2} g_{S^2} - \theta  \otimes \theta$.\par 
We will employ standard angular coordinates $(\vartheta, \varphi)$ on $S^2$, and work with the orthonormal coframe
$$e^1 = (2\kappa)^{-1}\, \rmd \vartheta\,, \qquad e^2 = (2\kappa)^{-1}\, \rmd \varphi \, \sin(\vartheta)\,.$$
The volume form on $P$ given by $\mathsf{Vol}_P = e^1 \wedge e^2  \wedge \theta$. Using the orthonormal coframe, the Hodge dual can easily be calculated.
\begin{align*}
& \star \mathbf{1} = e^1 \wedge e^2 \wedge \theta\,, && \star( e^1 \wedge e^2) = \theta\,,\\
& \star e^1 = e^2 \wedge \theta\,,&& \star(e^1 \wedge \theta)=e^2\,,\\
&\star e^2 =  - e^1 \wedge \theta\,, &&\star(e^2 \wedge \theta) = - e^1\,,\\
& \star \theta =  - e^1 \wedge e^2\,,&& \star(e^1 \wedge e^2 \wedge \theta) = - \mathbf{1}\,.
\end{align*}
We define the following differential forms using the stated orthonormal coframe
\begin{align*}
&f \in C^\infty(P) = \Omega_\sH^0(P) \,,
&& \alpha =   S_1 + \theta \wedge T_0 \in \Omega^1_\sH \oplus \theta \wedge \Omega_\sH^0(P)\,,\\
& \beta  = S_2  + \theta \wedge T_1 \in \Omega^2_\sH \oplus \theta \wedge \Omega_\sH^1(P) && \gamma =   \theta \wedge T_2 \in  \theta \wedge \Omega_\sH^2(P)\,,  
\end{align*}
so that a direct computation using computer algebra gives 
\begin{center}
\renewcommand{\arraystretch}{1.4}
\begin{tabular}{c|l}
\textbf{Form Degree} & \textbf{Hodge--de Rham Laplacian} \\ \hline
$0$ & $\Delta_{HdR} f =  \Delta_{S^2}f  -  \Delta_P^2 f$ \\
$1$ & \begin{minipage}[t]{0.85\textwidth}
$\Delta_{HdR} \alpha =  \big(\Delta_{S^2} S_1 - \mathcal{L}_{\Delta_P}^2 S_1 - 2\, \mathcal{L}_{\Delta_P} S_1\big) - \theta \wedge\big ( 2\,  \operatorname{div}_{S^2}(\vec{S}_1)-(\Delta_{S^2} T_0 - \mathcal{L}_{\Delta_P}^2 T_0) \big) $
\end{minipage} \\
$2$ & \begin{minipage}[t]{0.85\textwidth}
$\Delta_{HdR} \beta =\big(\Delta_{S^2} S_2 - \mathcal{L}_{\Delta_P}^2 S_2 - 2\, \mathcal{L}_{\Delta_P} S_2\big) - \theta \wedge\big ( 2\,  \operatorname{div}_{S^2}(\vec{T}_1)-(\Delta_{S^2} T_1 - \mathcal{L}_{\Delta_P}^2 T_1) \big)$
\end{minipage} \\
$3$ & $\Delta_{HdR} \gamma = - \theta \wedge  \big(\Delta_{S^2} T_2 - \mathcal{L}_{\Delta_P}^2 T_2\big)$
\end{tabular}
\end{center}
where $ \Delta_{S^2}$ is the standard Hodge--de Rham Laplacian acting on differential forms on $S^2$ with the scaled metric $(2 \kappa)^{-2} g_{S^2}$, and in angular coordinates
\begin{align*}
& \operatorname{div}_{S^2}( \vec{S}_1) := (2 \kappa)^2 \left(\frac{1}{\sin\vartheta} \frac{\partial}{\partial \vartheta} \left( \sin\vartheta\, S_{1,\vartheta} \right)
+ \frac{1}{\sin^2\vartheta} \frac{\partial S_{1,\varphi}}{\partial \varphi}\right)\,,\\ & \operatorname{div}_{S^2}( \vec{T}_1) := (2 \kappa)^2 \left(\frac{1}{\sin\vartheta} \rmd \vartheta \frac{\partial}{\partial \vartheta} \left( \sin\vartheta\, T_{1, \vartheta} \right)
+ \frac{1}{\sin^2\vartheta}\,\rmd \varphi \frac{\partial  T_{1, \varphi}}{\partial \varphi}\right)\,.
\end{align*}
Note that as written, the Hodge--de Rham Laplacians are written invariantly and so not depend on the coordinates or coframe employed. Thus, the expressions in the above table are invariant under the relevant bundle automorphims.
\begin{remark}
The Hodge--de Rham Laplacian on functions (zero-forms) is the same as the Laplacian presented in \cite{Bruce:2025b}, which was derived using a sigma model. This is completely expected from the general theory of Laplacians and their relation with non-linear sigma models. 
\end{remark}
The question of extending the definition of the Hodge--de Rham Laplacians to include $t =0$, i.e., being well-defined for all time, can be addressed. We will make the natural assumption that all the components of the differential forms are smooth for all $t$, in particular including $t =0$. Then, if the components $S_1, T_0, T_1$ and $T_2$ are at least linear in $t$ in the neighbourhood of the zero section, then the Hodge--de Rham Laplacians on the event horizon of a black hole extend to include $t=0$.  We will refer to such differential forms as \emph{regular differential forms}.
\begin{definition}
A regular differential form $\zx$ on $(S^2 \times \R^\times , (2 \kappa)^{-2} g_{S^2} )$  is said to be a \emph{harmonic form on the Schwarzschild horizon} if it is harmonic, i.e., $\Delta_{HdR} \zx =0$. 
\end{definition}  
%
%

\section{Concluding Remarks}
In this note, we have developed the fundamental elements of Hodge theory on a Carrollian $\R^\times$-bundle $(P, g,  \Phi)$. Specifically, we constructed the Hodge star operator, de Rham codifferential and Hodge--de Rham Laplacian. These structures can naturally be transplanted to the associated line bundle $L$ via the smooth inclusion $P \hookrightarrow L$. However, the Hodge star operator becomes degenerate at $t =0$, thus, care is required with the derived (maybe local) equations. \par 
None-the-less, the construction of a Hodge star operator allows for a version of electromagnetism to be geometrically formulated on a Carrollian $\R^\times$-bundle. The specific case of working on $P = \R^3\times \R^\times$ and the trivial connection,  the theory closely parallels the standard formulation of electromagnetism on a curved spacetime. An important property of the derived Carrollian Maxwell equations is that they can be extended to include $t=0$ using local coordinates. Interestingly, the Carrollian Maxwell equations take on the familiar form when using ``logarithmic time''. This is very suggestive that $u = \ln |t|$ is the physically relevant temporal degree of freedom in intrinsic Carrollian geometry. \par  
As demonstrated, via the example of the event horizon of a black hole, it is possible to have -- under some mild conditions -- a well-defined Hodge--de Rham Laplacian for all $t$ is possible. In particular, differential forms must satisfy a  natural regularity condition to ``absorb'' the problem of the connection being ill-defined near the zero section. \par

%
%

%
\end{document}